\documentclass[11pt, oneside]{article}

\usepackage{amsmath, amssymb, amsthm}
\usepackage{xcolor}
\usepackage{geometry, enumerate,tikz}
\usepackage{graphicx}
\usepackage[font=small,labelfont=bf, width=.85\textwidth]{caption}

\newtheorem{theorem}{Theorem}

\newtheorem{lemma}[theorem]{Lemma}
\newtheorem{example}[theorem]{Example}
\newtheorem{corollary}[theorem]{Corollary}

\newtheorem{conjecture}[theorem]{Conjecture}

\newtheorem{claim}{Claim}
\usepackage{comment}

\newcommand{\G}{\mathcal{C}}
\newcommand{\M}{\mathcal{M}}

\newcommand{\mmax}{M^+}
\newcommand{\mmin}{M^-}
\newcommand{\family}{\M_n}

\newenvironment{proofc}{\begin{proof}[Proof of Claim]}{\end{proof}}

\DeclareMathOperator{\bi}{bind}

\begin{document}

\title{Degree sequences realizing labelled perfect matchings}
\author{
Joseph Briggs\footnote{Auburn University, Department of Mathematics and Statistics, Auburn U.S.A.
  Email: {\tt jgb0059@auburn.edu}.}
\qquad
Jessica McDonald\footnote{Auburn University, Department of Mathematics and Statistics, Auburn U.S.A.
  Email: {\tt mcdonald@auburn.edu}.
	Supported in part by Simons Foundation Grant \#845698 and NSF grant DMS-2452103}
\qquad
Songling Shan\footnote{Auburn University, Department of Mathematics and Statistics, Auburn U.S.A.
  Email: {\tt szs0398@auburn.edu}.
	Supported in part by NSF grant
			DMS-2345869}
}

\date{}

\maketitle

\begin{abstract}
Let $n\in \mathbb{N}$ and $d_1 \geq d_2 \geq d_n\geq 1$ be integers. There is characterization of when $(d_1, d_1, \ldots, d_n)$ is the degree sequence of a graph containing a perfect matching, due to results of Lov\'{a}sz (1974) and Erd\H{o}s and Gallai (1960). But \emph{which} perfect matchings can be realized in the labelled graph? Here we find the extremal answers to this question, showing that the sequence $(d_1, d_2, \ldots, d_n)$: (1) can realize a perfect matching iff it can realize $\{(1, n), (2,n-1), \ldots, (n/2, n/2+1)\}$, and; (2) can realize any perfect matching iff it can realize $\{(1, 2), (3,4), \ldots, (n-1, n)\}$. Our main result is a  characterization of when (2) occurs, extending the work of Lov\'{a}sz and Erd\H{o}s and Gallai. Separately, we are also able to establish a conjecture of Yin and Busch, Ferrera, Hartke, Jacobsen, Kaul, and West about packing graphic sequences, establishing a degree-sequence analog of the Sauer-Spencer packing theorem. We conjecture an $h$-factor analog of our main result, and discuss implications for packing $h$ disjoint perfect matchings.
\end{abstract}

\section{Introduction}

All graphs in this paper are assumed to be simple.  

Consider a weakly decreasing sequence $d_1 \geq \dots \geq d_n$ of non-negative integers for some $n\in\mathbb{N}$. If there exists a graph $G$ where $V(G)=[n]=\{1, 2, \ldots, n\}$ and $\deg_G(i)=d_i$ for $1\leq i\leq n$, then we say that $(d_1, \ldots, d_n)$ is \emph{graphic}, and that $G$ \emph{realizes} $(d_1, \ldots, d_n)$; we also refer to  $(d_1, \ldots, d_n)$ as the \emph{degree sequence} of $G$. We say that $(d_1, \ldots, d_n)$ \emph{can realize a perfect matching} if there exists a graph $G$ which realizes $(d_1, \ldots, d_n)$ and such that $G$ contains a perfect matching. The following two theorems together answer the question of when  $(d_1, \ldots, d_n)$  can realize a perfect matching.

\begin{theorem}[Lov\'asz \cite{Lov}] \label{Lov} Let $n \in \mathbb{N}$ and  $d_1 \geq \dots \geq d_n$ be non-negative integers. Then $(d_1, d_2, \ldots, d_n)$ can realize a perfect matching iff $(d_1, d_2, \ldots, d_n)$ and $(d_1-1, \ldots, d_n-1)$ are both graphic.
\end{theorem}

\begin{theorem}[Erd\H{o}s--Gallai \cite{EG}]\label{EG}
   Let  $n\in \mathbb{N}$ and $d_1 \geq d_2 \geq \dots \geq d_n$ be non-negative integers. The sequence $(d_1, d_2, \ldots, d_n)$ is graphic iff $\sum_{i=1}^n d_i$ is even and for every $k\in[n]$:
    \begin{equation}\label{EGsum}
    \sum_{i=1}^k d_i \leq k(k-1) + \sum_{i=k+1}^n \min\{d_i,k\}.
    \end{equation}
\end{theorem}

The above theorems tell us when $(d_1, d_2, \ldots, d_n)$ can realize a perfect matching, but they do not say anything about \emph{which} perfect matchings are realized in the labelled graph. Given the complete graph $K_n$ on the vertex set $[n]$ (with $n$ even) we denote by $\family$ the set of all perfect matchings in $K_n$. We use the notation $(i, j)$ to refer to the edge between vertex $i$ and vertex $j$, so in particular  the following are two different elements of $\family$:
\begin{eqnarray*}
    M^-:&=&\{(1, n), (2, n-1), \ldots, (n/2, n/2+1)\};\\
     M^+:&=&\{(1, 2), (3, 4), \ldots, (n-1, n)\}.
\end{eqnarray*}
Given a graph $G$ on the vertex set $[n]$, we may ask if a particular $M\in \family$ exists in $G$. We may also ask when $(d_1, d_2, \ldots, d_n)$ \emph{can realize some particular} $M\in\family$, that is, when there exists a graph $G$ with $V(G)=[n]$, $\deg_G(i)=d_i$, and $G$ contains the perfect matching $M$.

In this paper we find that the perfect matchings $M^{-}$ and $M^{+}$ are extremal in terms of being realizable.

\begin{theorem}\label{t.extremal} Let  $n\in \mathbb{N}$ and $d_1 \geq d_2 \geq \dots \geq d_n\geq 1$ be integers.
\begin{enumerate}
\item[(a)] The sequence $(d_1, d_2, \ldots, d_n)$ can realize a perfect matching iff it can realize $M^{-}$.
\item[(b)] The sequence $(d_1, d_2, \ldots, d_n)$ can realize any $M\in\family$ iff it can realize $M^+$.
\end{enumerate}
\end{theorem}

Theorems \ref{Lov} and \ref{EG} together provide a characterization of when (a) occurs. Our main result in this paper is the following extension of this earlier work of Lov\'asz and Erd\H{o}s--Gallai, which characterizes when (b) occurs.

\begin{theorem}\label{t.SpecialMatch} Let  $n\in \mathbb{N}$ and $d_1 \geq d_2 \geq \dots \geq d_n\geq 1$ be integers. The sequence $(d_1, d_2, \ldots, d_n)$ can realize $M^+$ iff $\sum_{i=1}^n d_i$ is even, $n$ is even, and for every $k\in[n]$:
\begin{equation}\label{EGsumMatch}\tag{*}
\mathop{\sum_{i=1}^k} d_i \leq \begin{cases}
      k(k-1)+\mathop{\sum_{i=k+1}^n} \min\{d_i-1, k\} & \textrm{if $k$ is even}\\
      k(k-1)+\min\{d_{k+1}, k\}+\mathop{\sum_{i=k+2}^n} \min\{d_i-1, k\} & \textrm{if $k$ is odd}.
   \end{cases}
\end{equation}
\end{theorem}

Our proof of Theorem \ref{t.SpecialMatch} is inductive (in particular, it is inspired by a proof of Theorem \ref{EG} due to Choudum \cite{Ch}) and leads to a polynomial-time algorithm for constructing a realization of $M^+$. Our argument for Theorem \ref{t.extremal}(b) also provides a polynomial-time algorithm that transforms one realization to another. So, meeting the conditions of Theorem \ref{t.SpecialMatch} means that not only does a realization exist for any $M\in\family$, but that we can construct it in polynomial time.

As a corollary to Theorem \ref{t.SpecialMatch}, we can obtain the following sufficient condition for realizing $M^+$.

\begin{corollary}\label{corMatch} Let $(d_1, d_2, \ldots, d_n)$ be a graphic sequence, with $n$ an even integer and $d_n\ge n/2$.
If $\sum_{i=1}^{n/2}d_i \le (\sqrt{2-(4/n)}-0.5) \tfrac{n^2}{2}$,
 then the sequence $(d_1, d_2, \ldots, d_n)$ can realize $\mmax$.
\end{corollary}

Recall that every $n$-vertex graph $G$ with $\delta(G)\geq \tfrac{n}{2}$ and $n$ even has a perfect matching (e.g. apply Driac's Theorem to get a Hamilton cycle and then take every second edge in the cycle). When $n$ is large, the extra condition in Corollary \ref{corMatch} is approximately $\sum_{i=1}^{n/2}d_i \le (.9)\tfrac{n^2}{2}$, i.e., we want the first half of the degree sum to be no more than about 90\% of what it could possibly be (if all these degrees were $n-1$). This bound is asymptotically best possible: we will later demonstrate a degree sequence which just barely fails this condition but cannot realize $\mmax$.

Two graphic sequences $(d_1^1, \ldots, d_n^1)$ and $(d_1^2, \ldots, d_n^2)$ are said to \emph{pack} if there are edge-disjoint graphs $G_1$ and $G_2$
on the same vertex set $[n]$ such that $\deg_{G_j}(i)=d_i^j$ for all $i\in [n]$ and $j\in[2]$.  Two $n$-vertex graphs $G_1$  and $G_2$ \emph{pack} if they can be expressed as edge-disjoint subgraphs of the complete graph $K_n$. In 1978, Sauer and Spencer~\cite{MR516262} established that $n$-vertex graphs  $G_1$ and $G_2$  pack if $\Delta(G_1)\Delta(G_2)<n/2$.
%Using Corollary \ref{corMatch},
We are able to fully verify the following conjecture, establishing a degree sequence analog of the Sauer-Spencer packing theorem.

\begin{conjecture}[Yin \cite{Yin}; Busch, Ferrera, Hartke, Jacobsen, Kaul, and West~\cite{busch2012packing}]\label{pack}
   Let $n\ge 3$ be an  integer, and let $(d_1^1, \ldots, d_n^1)$ and $(d_1^2, \ldots, d_n^2)$
   be two graphic sequences with $d_n^1, d_n^2\ge 1$.  If $d_1^1 d_1^2  <\frac{n}{2}$, then $(d_1^1, \ldots, d_n^1)$ and $(d_1^2, \ldots, d_n^2)$  pack.
\end{conjecture}

Conjecture \ref{pack} was first posed by Busch, Ferrera, Hartke, Jacobsen, Kaul, and West ~\cite{busch2012packing} with  $d_n^1+ d_n^2\ge 1$ in place of
$d_n^1, d_n^2\ge 1$. Yin \cite{Yin} disproved this initial version, but proposed replacing $d_n^1+ d_n^2\ge 1$
with $d_n^1, d_n^2\ge 1$, which we can now say is indeed sufficient.

Our paper proceeds as follows. We prove Theorem \ref{t.SpecialMatch} in Section 2. Corollary \ref{corMatch} is proved in Section 3, where the above-mentioned tightness example is also given. We confirm Conjecture \ref{pack} in Section 4. Note that this confirmation requires a result on the  \emph{binding number} of a graph by Kang and Tokushige \cite{KT}, which will be discussed in Section 4. Section 5 contains a proof of Theorem \ref{t.extremal}, which we handle using the language of posets. We provide detailed examples for the family $\family$ in this section, and make a poset conjecture about the family. In the 6th and final section of this paper we conjecture a generalization of our main result, Theorem \ref{t.SpecialMatch}, from the realm of perfect matchings to that of $h$-factors (spanning $h$-regular subgraphs) for any  $h\in \mathbb{N}$. We show that a direct analog of Theorem \ref{t.extremal} for $h$-factors is false. On the other hand, we show that if our $h$-factor conjecture is true, then it implies a characterization for packing $h$ disjoint perfect matchings.

\section{Proof of Theorem \ref{t.SpecialMatch}}

For convenience we restate our main theorem again here.

\setcounter{theorem}{3}
\begin{theorem} Let  $n\in \mathbb{N}$ and $d_1 \geq d_2 \geq \dots \geq d_n\geq 1$ be integers. The sequence $(d_1, d_2, \ldots, d_n)$ can realize $M^+$ iff $\sum_{i=1}^n d_i$ is even, $n$ is even, and for every $k\in[n]$:
\begin{equation}\label{EGsumMatch}\tag{*}
\mathop{\sum_{i=1}^k} d_i \leq \begin{cases}
      k(k-1)+\mathop{\sum_{i=k+1}^n} \min\{d_i-1, k\} & \textrm{if $k$ is even}\\
      k(k-1)+\min\{d_{k+1}, k\}+\mathop{\sum_{i=k+2}^n} \min\{d_i-1, k\} & \textrm{if $k$ is odd}.
   \end{cases}
\end{equation}
\end{theorem}
\setcounter{theorem}{6}

\begin{proof}
Suppose first that $(d_1, \ldots, d_n)$ realizes $\mmax$ via the graph $G$. Since $G$ is graphic we must have $\sum_{i=1}^n d_i$ is even, and since $G$ has a 1-factor $n$ must be even. Consider the graph $G'$ obtained from $G$ by deleting the all the edges $(k+1, k+2), \ldots, (n-1, n)$ (if $k$ is even), or by deleting all the edges $(k+2, k+3), \ldots, (n-1, n)$ (if $k$ odd). Then $G'$ must satisfy condition (\ref{EGsum}) of Theorem \ref{EG} for $k$. For both $k$ even and $k$ odd, this amounts precisely to (\ref{EGsumMatch}).

We proceed with the backwards direction of our proof by induction on $\sum_{i=1}^n d_i$. Since $d_n\geq 1$ this sum is at least $n$. If it is equal to $n$, then the desired 1-factor is realized by the graph consisting of just the 1-factor itself. So we may assume that $\sum_{i=1}^n d_i>n$. Since $n$ is even and $\sum_{i=1}^n d_i$ is also even, we in fact get that $\sum_{i=1}^n d_i\geq n+2$. This extra $+2$ must be split among at least two $d_i$'s: if not, our sequence is $(d_1, 1, \ldots, 1)$ and applying (\ref{EGsumMatch}) with $k=1$ tells us that $d_1\leq 0+1+0$, a contradiction.

Let $p$ be the largest integer such that $d_p\geq 2$.
If $d_1=d_2=\cdots=d_p$ then $p\geq 2$ and we may let $t=p-1$. Otherwise,
%let $t$ be the smallest integer such that $d_t>d_{t+1}$.
at least one $t<p$ has $d_t > d_{t+1}$, so let $t$ be the largest such.
Now consider the following sequence:
$$(d_1, \ldots, d_{t-1}, d_{t}-1, d_{t+1}, \ldots, d_{p-1}, d_p -1, d_{p+1}, \ldots, d_n).$$
Note that by our choice of $t$ and $p$, $$d_1 \geq \cdots \geq d_{t-1} >  d_{t}-1 \geq  \underline{d_{t+1} =  \cdots = d_{p-1} >} d_p-1\geq d_{p+1}\geq \cdots  \geq d_n,$$
where the underlined subsequence may be empty (if $t=p-1$).
We shall refer to our original sequence $(d_1, \ldots, d_n)$ as $\pi$ and to the above sequence as $\pi'$.

\begin{claim}\label{tp} $d_t
\leq p-1$.
\end{claim}
\begin{proofc} If not, this implies
$$
\sum_{i=1}^t
 d_i \geq pt = t(t-1)+2t +(p-t-1) \cdot t>t(t-1) + \min\{d_{t+1},t\} +\sum_{i=t+2}^p \min\{d_i-1,t\},
 $$
contradicting the fact that (\ref{EGsumMatch}) holds for $\pi$.
\end{proofc}

\begin{claim}\label{badk} We may assume that there exists $k\in[n]$ for which (\ref{EGsumMatch}) fails for $\pi'$.
\end{claim}
\begin{proofc}
If not, we may apply induction to $\pi'$. This tells us that it can realize the desired matching; let $G'$ be such a realization on $[n]$. If $t \not\sim p$ in $G'$, then by adding the edge $(t, p)$ we get the graph $G$ that we need. So we may assume that $(t, p)\in E(G')$.

Note that $\deg_{G'}(t)=d_t-1 \leq p-2$, by Claim \ref{tp}. This means that there exists $x \in V(G')$ such that $x\not\sim t$ in $G'$ and $x\leq p$. But then $\deg_{G'}(x)\geq \deg_{G'}(p)$, while $t$ contributes to the degree of $p$ but not to the degree of $x$, this means that there exists $y$ such that $y\sim x$ but $y\not\sim p$ in $G'$. We now have a square of vertices $(x, y, p, t)$ with one matching in $G'$ ($(t, p), (x, y)$) and another matching ($(x, t), (y, p)$) not in $G'$; by switching and then adding the edge $(t, p)$ we get a graph $G$ with our desired degree sequence. This $G$ is as desired unless the edge $(x,y)$, the only edge of $G'$ that is not in $G$, happens to be one of our matching edges. If $\deg_{G'}(x)\geq \deg_{G'}(p)+1$, then we could have avoided this scenario as there were at least 2 valid choices for $y$ above. But this is certainly true because $\deg_{G'}(x)=d_x\geq d_p>d_{p}-1=\deg_{G'}(p)$.
\end{proofc}

We will now work to narrow down the possible values for $t, k, p$ and $d_t$. In most situations we will show that the above induction works: namely, we'll contradict Claim \ref{badk} by showing that $\pi'$ actually satisfies (\ref{EGsumMatch}) for $k$. There will also be a couple of instances where we can't get that contradiction (i.e.\ induction does not work), but in those situations we will be able to describe explicit constructions for the desired $G$.

\begin{claim}\label{case1} $t\geq k+1$ (and $p\geq k+2$).
\end{claim}
\begin{proofc}
Suppose, for a contradiction, that $t\leq k$. Consider the truth of (\ref{EGsumMatch}) for our original sequence. If $p\leq k$ then we need only subtract from the right-hand-side of the inequality to get (\ref{EGsumMatch}) for our new sequence, contradicting Claim \ref{badk}. Otherwise since $t\leq k$ by subtracting one from both sides we get (\ref{EGsumMatch}); note that on the right-hand-side we may fold the -1 into the sum for just $d_p$.
\end{proofc}

\begin{claim}\label{tk} $d_t \geq k$.
\end{claim}
\begin{proofc}
\indent By our choice of $t$, we know that $d_1=\cdots =d_t$, so we get that  $d_1=\cdots =d_k \leq k-1$. Hence we get that
$$\sum_{i=1}^{k} d_i \leq k(k-1),$$
which immediately implies (\ref{EGsumMatch}) for our new sequence, contradicting Claim \ref{badk}.
\end{proofc}

\begin{claim}\label{parity} If $d_t=k$, then the two sides of (\ref{EGsumMatch}) have the same parity for $\pi'$ (and for $\pi$). The same is true when $d_t=k+1$ and $k$ is even.
\end{claim}

\begin{proofc}
First observe that $\sum_{i=1}^k d_i$ has the same parity as $\sum_{i=k+1}^n d_i$ since $\sum_{i=1}^n d_i$ is even. This in turn has the same parity as $\sum_{i=k+1}^n (d_i -1)$ when $k$ is even (since $n-k$) is even, and has the same parity as $d_{k+1}+ \sum_{i=k+2}^n (d_i -1)$ when $k$ is odd (since $n-(k+1)$) is even.  Since $\min\{d_i-1, k\}=d_i-1$ for all $i$, we get that $\sum_{i=1}^k d_i$ has the same parity as $\sum_{i=k+1}^n \min\{d_i -1,k\}$ when $k$ is even, and has the same parity as $d_{k+1}+ \sum_{i=k+2}^n \min\{d_i -1,k\}$ when $k$ is odd.

Note that $\min\{d_i-2, k\}=\min\{d_i-1, k\}-1$ for $i\in t, p$, since $d_t\leq k+1$. Since subtracting two does not change the parity, and since $k(k-1)$ is always even, this immediately gives our desired result for $k$ even.

Now suppose that $k$ is odd. If $d_{k+1}=k$ then we get our desired result, as in the even case.
\end{proofc}

\begin{claim}\label{tk+1} We may assume that $d_t\geq k+1$.
\end{claim}
\begin{proofc}
By our choice of $t$ we get that $d_1=\cdots d_t=k$ and since $t>k$ this means $d_1=\cdots d_{k+1}=k$. We get
\begin{equation}\label{case3sum}
    \sum_{i=1}^{k} d_i = k(k-1)+ k= k(k-1)+ \min\{d_{k+1}, k\}.
\end{equation}
Suppose first that $k$ is odd.  By Claim \ref{parity}, it is enough to show that $\left(\sum_{i=k+2}^n (d_i-1)\right)\geq 1$, since this means we could add $\left(\sum_{i=k+2}^n (d_i-1)\right)-1$ to the right-hand-side of (\ref{case3sum}) and be within one of our desired inequality, yielding a contradiction to Claim \ref{badk}.
But since $k+2\leq p$ we get our result immediately since $d_p\geq 2$.

We may now assume that $k$ is even.  We consider the quantity
\begin{equation}\label{RHS}
k(k-1)+\mathop{\sum_{i=k+1, i\neq t, p}^n} \min\{d_i-1, k\}+ \min\{d_t-2, k\}+\min\{d_p-2, k\},
\end{equation}
which we aim to show is greater than or equal to (\ref{case3sum}), in order to contradict Claim \ref{badk}. In particular, given Claim \ref{parity}, it suffices to show that the last three terms in (\ref{RHS}) are at least $k-1$. If $t\ge k+2$, then the sum of the two middle terms is at least
$k-1+k-2=2k-3$, which is at least $k-1$ since $k\geq 2$. So we may assume that $t=k+1$. Now $\min\{d_{t}-2, k\}=k-2$, and we still get our desired result if $d_p\geq 3$, or if there is a another contributing term in our sum, i.e., if $p\ge k+3$. So we may assume that $p=k+2$ and $d_p=2$, meaning that $\pi$ is $(k, \ldots, k, 2, 1, \ldots, 1)$, with $k+1$ leading $k$'s, followed by one two, and then some even number of 1's (since $k$ is even). In this case, we construct our desired $G$
by taking a complete graph on $[k+1]$,  removing the edge joining $k$ and $k+1$,
adding one edge joining $k$ and $k+2$ and one edge joining $k+1$ and $k+2$,  and for each odd $j\in [k+3, n]$, we
add an edge joining $j$ and $j+1$.
\end{proofc}

\begin{claim} We may assume that (\ref{EGsumMatch}) is tight for $\pi$ at value $k$.
\end{claim}

\begin{proofc}
If instead of tightness there is a slack of two or greater, then we may freely subtract two from the right-hand-side, which gives us a contradiction to Claim \ref{badk}. So we may assume there is a slack of exactly one. This means we can freely subtract one from the right-hand side, allowing us to replace $\min\{d_p-1, k\}$ with $\min\{d_p-2, k\}$. If we also know that $d_{t}\geq k+2$, then we again get a contradiction to Claim \ref{badk}. In fact, we also get this contradiction when $t=k+1$ if $d_t=k+1$ and $t$ is odd. So we may assume that $d_1=\cdots =d_{t}= k+1$, and either $t\geq k+2$ or $t=k+1$ and $k$ is even. However, since $d_t=k+1$, if $k$ is even then Claim \ref{parity} tells us that the two sides of (\ref{EGsumMatch}) have the same parity for $\pi$, while our assumption is that they differ by exactly one.  So it must be the case that $t\geq k+2$ and $k$ is odd.

We know that the left-hand side of (\ref{EGsumMatch}) for $k$ and $\pi'$ is $k(k+1)=k^2+k$, while the sum on the right-hand side of (\ref{EGsumMatch}) for $k$ and $\pi'$ contains $k(k-1)=k^2-k$ as a term, and also contains $\min\{d_t-2, k\}=k-1$, $\min\{d_p-1, k\}$, and either $\min\{d_{k+1},k\}=k$ (since $k$ is odd); hence we get at least $k^2+k-1$. We get at least one more than this (and hence achieve our desired contradiction), if $t\ge k+3$,  if $p\ge t+2$,  or if $d_p \ge 3$. So we may assume that $d_1= \ldots =d_{k+2}=k+1$, $d_{k+3}=2$, and $d_i=1$ for each $i\in \{k+4, \ldots, n\}$. In this case, we construct $G$ by taking a complete graph $[k+2]$,  removing the edge joining $k+1$ and $k+2$, adding edges to join both $k+1$ and $k+2$ to $k+3$, and then, for each odd $j\in \{k+4, \ldots,  n\}$, we add an edge joining $j$ and $j+1$.
\end{proofc}

Following the above claim we are able to write an equality in (\ref{EGsumMatch}) for $\pi$ at value $k$.  In order to write this in a more convenient form, let $r$ be the largest index such that $d_{r} \ge k+1$; by Claims \ref{tk+1} and \ref{case1} we know that $r\geq t\geq k+1$. We get
\begin{eqnarray}
&& k\cdot d_t=k(k-1)+k(r-k) +\sum_{i=r+1}^n (d_i-1) \label{kdt}\\
&\Rightarrow& d_t=r-1+\frac{1}{k}\sum_{i=r+1}^n (d_i-1). \label{case5}
\end{eqnarray}

\begin{claim}\label{last} $d_t=k+1$ and $k+1$ is odd.
\end{claim}

\begin{proofc} First suppose that $d_t=k+1$ and $k+1$ is even.
We know that $\pi$ also satisfies (\ref{EGsumMatch}) for $k+1$, and since $d_t= k+1$ and $k+1$ is even, this says that
\begin{equation}\label{pik+1}(k+1)d_t \le k(k+1)+k(r-k-1)+\sum_{i=r+1}^n (d_i-1).
\end{equation}
Subtracting (\ref{kdt}) from (\ref{pik+1}) yields $d_t\leq 2k-k=k$, contradicting our assumption.

We may now assume, by Claim \ref{tk+1}, that $d_t\geq k+2$. Let $s$ be the largest index such that $d_{s} \ge k+2$, and consider that $\pi$ also satisfies (\ref{EGsumMatch}) for $k+1$. From this we get the following (noting that $\min\{d_{k+1}, k+1\}=\min\{d_{k+1}-1, k+1\}=k+1$ since $d_{k+1}=d_t\geq k+2$) :
\begin{equation}\label{spik+1}(k+1)d_t \le (k+1)k+(k+1)(s-k-1)+ k(r-s)+\sum_{i=r+1}^n (d_i-1).
\end{equation}
Subtracting~\eqref{case5} from~\eqref{spik+1} gives:
\begin{equation}\label{s-1}
d_t\le 2k -k+s-k-2=s-1.
\end{equation}
Combining this with (\ref{case5}), we get
$ s\ge r+\frac{1}{k}\sum_{i=r+1}^n (d_i-1)$,  a contradiction to $s\le r$ unless
$\frac{1}{k}\sum_{i=r+1}^n (d_i-1)=0$.
In this case, we have $s=r$ and $d_{r+1} =\ldots =d_n=1$, and so $d_p =d_r\ge k+2$.
In this case, the right hand side of (\ref{EGsumMatch}) is the same for both $\pi$ and $\pi^*$, contradicting Claim \ref{badk}.
\end{proofc}

By Claim \ref{last} we now know that $d_t=k+1$ and $k+1$ is odd.  This implies that $d_1= \ldots =d_{k+1} =k+1$, and $d_{k+2} \le k$.  From (\ref{kdt}) we get
 $$k(k+1)=k(k-1)+k +\sum_{i=k+2}^n (d_i-1),$$
 which in particular implies that $\sum_{i=k+2}^n (d_i-1)=k$. In this case, we construct $G$ directly as follows. We start with a complete graph on $[k+1]$, and add an edge joining $k+1$ and $k+2$, and then add an edge between  $i$ and $i+1$ for every  odd $i\in \{k+3, \ldots n\}$. Then, for any $i\in \{k+2,\ldots, n\}$, we  add $d_i-1$ edges  joining $i$ and $d_i-1$  distinct vertices
from $[k]$. Since  $\sum_{i=k+2}^n (d_i-1)=k$ and there are exactly $k$ vertices in [k] that each need exactly one more degree, this is possible, and gives us our desired $G$.
\end{proof}

\section{Proof of Corollary~\ref{corMatch}}

We restate Corollary~\ref{corMatch} for convenience.

\setcounter{theorem}{4}
\begin{corollary} Let $(d_1, d_2, \ldots, d_n)$ be a graphic sequence, with $n$ an even integer and $d_n\ge n/2$.
If $\sum_{i=1}^{n/2}d_i \le (\sqrt{2-(4/n)}-0.5) \tfrac{n^2}{2}$,
 then the sequence $(d_1, d_2, \ldots, d_n)$ can realize $\mmax$.
\end{corollary}
\setcounter{theorem}{6}

\begin{proof}
Let $\pi=(d_1, d_2, \ldots, d_n)$ be a graphic sequence, with $n$ an even integer and $d_n\ge n/2$.
It suffices to show that  $\pi$ satisfies condition~\eqref{EGsumMatch}. Let $k\in [n]$ be any integer.
As $d_n\ge n/2$, it is easy to see that \eqref{EGsumMatch} is satisfied when $k \le n/2-1$: in this case the right-hand-side of \eqref{EGsumMatch} is exactly $k(k-1)+(n-k)k=k(n-1)$, which is trivially an upper bound on the left-hand-side of \eqref{EGsumMatch}. Thus we assume $k\ge n/2$. It suffices to show that
$$\sum_{i=1}^k d_i \le k(k-1)+(n-k)(n/2-1).$$
Let $\overline{d}=\left(\sum_{i=1}^{n/2} d_i\right)/(n/2)$. Then $\overline{d}\geq \left(\sum_{i=1}^{k} d_i\right)/k$, since $k\ge n/2$. So it suffices to prove that
$$\overline{d}k \le k(k-1)+(n-k)(n/2-1).$$
Let $f(k)=k^2-(\overline{d}+n/2)k+n(n/2-1)$. Then $f(k)$ is a quadratic and concave up function, with a minimum at $k_0=\frac{1}{2}(\overline{d}+n/2)$. It remains only to show that $f(k_0)\geq 0$, to which end we compute as follows:
\begin{eqnarray*}
    f(k_0)&=& k_0^2-2k_0^2+n(\tfrac{n}{2}-1) \\
     & =&-k_0^2+n(\tfrac{n}{2}-1)\\
     &=& -\tfrac{1}{4}\overline{d}^2-\tfrac{n}{4}\overline{d}+(\tfrac{7n^2}{16}-n).
\end{eqnarray*}
Using the quadratic formula we find that $f(k_0)\geq 0$ provided that $$\overline{d} \leq n(\sqrt{2-(4/n)}-0.5),$$
which corresponds exactly to our condition on $\sum_{i=1}^{n/2} d_i$.
\end{proof}

\vspace*{.1in}

We now provide an example to show that the bound of $$\sum_{i=1}^{n/2}d_i \le (\sqrt{2-(4/n)}-0.5) \tfrac{n^2}{2}$$ in Corollary \ref{corMatch} is best possible up to an additive constant. To this end, let $n>2$ be an even integer and define:
$$d^*=\left\lfloor (\sqrt{2}-\tfrac{1}{n}-0.5)n\right\rfloor$$
and
$$k^*=\left\lfloor\tfrac{1}{2}(d^*+n/2+1)\right\rfloor.$$
Consider the sequence
$$\pi^*=(d^*, d^*, \ldots, d^*, n/2, n/2, \ldots, n/2),$$
whose initial $k^*$ entries are all $d^*$ and remaining $n-k^*$ entries are all $n/2$. Since $n\geq 3$ we get that $d^*, k^*\geq n/2$. We will show that $\pi^*$ is a degree sequence that does not realize $\mmax$. This will give our desired example, since the sum of the first $n/2$ entries in $\pi^*$ is $d^* \cdot \tfrac{n}{2}$
$= (\sqrt{2-(4/n)} - 0.5 ) \frac{n^2}{2}-O(1)$.

We first show that $\pi^*$ is a degree sequence. To this end it suffices to show that $\pi^*$ satisfies the Erd\H{o}s-Gallai condition (i.e.\ (\ref{EGsum}) of Theorem \ref{EG}) for all $k$. When $k\leq n/2$ the right-hand side of the condition is exactly $k(k-1)+(n-k)k=k(n-1)$ which is trivially an upper bound on the left-hand side, so (\ref{EGsum}) holds. When $k > n/2$ it suffices to show that
$$kd^* \le k(k-1)+(n-k)(n/2).$$
Equivalently, we show that $f(k):=k^2-(d^*+n/2+1)k+n^2/2 \ge 0$.
Then $f(k)$ is a quadratic and concave up function, with a minimum at
$k_0:=\frac{1}{2}(d^*+n/2+1)$ (very similarly to the minimum of $f(k)$ in the proof of Corollary \ref{corMatch}). Since $f$ need only take integer values however, it suffices for us to to verify that $f(\lfloor k_0\rfloor)\geq 0$ (noting that $f(\lfloor k_0\rfloor)=f(\lceil k_0\rceil)$), and we need this extra degree of precision here. Let $\beta =k_0-\lfloor k_0\rfloor$. Then:
\begin{eqnarray*}
    f(\lfloor k_0\rfloor)=f(k_0-\beta) &=& (k_0-\beta)^2-2k_0(k_0-\beta) +n^2/2 \\
     &=& -k_0^2+\beta^2 +n^2/2 =-(\tfrac{1}{2}(d^*+n/2+1))^2+\beta^2 +n^2/2.
\end{eqnarray*}
Since $d^*=\left\lfloor (\sqrt{2}-\tfrac{1}{n}-0.5)n\right\rfloor,$ we let $\alpha =n(\sqrt{2}-\tfrac{1}{n}-0.5) -\left\lfloor n(\sqrt{2}-\tfrac{1}{n}-0.5)\right\rfloor$. Then we can continue computing as follows:
\begin{eqnarray*}
     f(\lfloor k_0\rfloor) &=& -\tfrac{1}{4}\left(n(\sqrt{2}-\tfrac{1}{n}-0.5)-\alpha +n/2+1\right)^2+n^2/2 +\beta^2\\
     &=& \tfrac{1}{4} (n\sqrt{2}-\alpha)^2+\beta^2 \ge 0,
\end{eqnarray*}
as desired.

We now show that $\pi^*$ does not satisfy~\eqref{EGsumMatch} by showing that \eqref{EGsumMatch} fails for $k=k^*$. This means we must show that
$$k^*d^*> k^*(k^*-1)+(n-k^*)(n/2-1)+1$$
if $k^*$ is odd, and the same inequality with one less on the right-hand side if $k^*$ is even.  So it suffices to show that
$$0> (k^*)^2-k^* (d^*+\tfrac{n}{2})+n(\tfrac{n}{2}-1)+1.$$
Let $g(k^*)=(k^*)^2-k^* (d^*+\tfrac{n}{2})+n(\tfrac{n}{2}-1)+1$, let $k_1 =\frac{1}{2}(d^*+n/2) $ and let $k^*=k_1+\gamma$ for some $-1<\gamma<1$. Then:
\begin{eqnarray*}
    g(k^*) =g(k_1+\gamma)&=& (k_1+\gamma)^2-(k_1+\gamma)(2k_1)+n(\tfrac{n}{2}-1)+1\\
    &=&-k_1^2+\gamma^2 +\tfrac{n^2}{2}-n+1 \\
    &=&-\tfrac{1}{4}(d^*+n/2)^2 +\gamma^2 +\tfrac{n^2}{2}-n+1\\
    &=& -\tfrac{1}{4}\left(n(\sqrt{2}-\tfrac{1}{n}-0.5)-\alpha+n/2\right)^2 +\gamma^2 +\tfrac{n^2}{2}-n+1\\
    &=& -\tfrac{1}{4}\left(n\sqrt{2}-(\alpha+1)\right)^2 +\gamma^2 +\tfrac{n^2}{2}-n+1\\
     &=&  -\tfrac{1}{4}(\alpha+1)^2 +\tfrac{\sqrt{2}}{2}(\alpha+1)n +\gamma^2-n+1.
\end{eqnarray*}
There are infinitely many values of $n\geq 3$ for which $\alpha \leq \tfrac{1}{4}$. For those choices of $n$, we get  that
$$g(k^*) < (.9)n -n+2\leq 0,$$
as desired.

\section{Packing of graphic n-tuples}

The binding number $\bi(G)$ of a graph $G$ is  defined as
$$\min \left\{\frac{|N_G(X)|}{|X|}: X\subseteq V(G), N_G(X) \ne V(G) \right\}.
$$
For a graph $G$ and a function $f: V(G)\rightarrow \mathbb{N}$, an \emph{$f$-factor} of $G$ is a spanning subgraph $H$ of $G$ satisfying $d_H(v)=f(v)$ for each $v\in V(G)$. The following result will be helpful for us.

\begin{theorem}[Kano and Tokushige \cite{KT}]\label{thm:f-factor}
Let $a$ and $b$ be integers such that $1\le a\le b$ and $b\ge 2$, and let $G$ be a connected simple graph with order $n$ with $n\ge \frac{(a+b)^2}{2}$. Let $f:  V(G)  \rightarrow  [a,b]$  be a function such that
$\sum_{v\in V(G)}f(v) \equiv 0 \pmod{2}$. If one of the following two conditions is satisfied,
then G has an $f$-factor.
\begin{itemize}
    \item  $\bi(G) \ge \frac{(a+b-1)(n-1)}{an-(a+b)+3}$.
    \item $\delta(G) \ge \frac{bn-2}{a+b}$.
\end{itemize}
\end{theorem}

We confirm Conjecture \ref{pack} as follows.

\begin{theorem}
   Let $n\ge 3$ be an integer, and let $(d_1^1, \ldots, d_n^1)$ and $(d_1^2, \ldots, d_n^2)$
   be two graphic sequences with $d_n^1, d_n^2\ge 1$.  If $d_1^1 d_1^2  <\frac{n}{2}$, then
   $(d_1^1, \ldots, d_n^1)$ and $(d_1^2, \ldots, d_n^2)$  pack.
\end{theorem}

\begin{proof}
By symmetry, we may assume that $d_1^1\ge d_1^2$. Let $G$ be a realization of $(d_1^1, \ldots, d_n^1)$ on $\{v_1, \ldots, v_n\}$.

Consider first the case that $d_1^2=\cdots = d_n^2=1$. Here it suffices to show that $\overline{G}$, which has degree sequence $(n-d_1^1, \ldots, n-d_n^1)$, contains a 1-factor. Note that $n$ is even, since $(d_1^2, \ldots, d_n^2)=(1, \ldots, 1)$ is graphic. Since $d_1^1<n/2$, the minimum degree in $\overline{G}$ is at least $n/2$, so we know that $G$ has a 1-factor by a corollary of Theorem \ref{Lov} and Theorem \ref{EG} (as discussed in the introduction).

We may now assume that $d_1^2\ge 2$. It suffices to show that $\overline{G}$
has an $f$-factor such that $f(v_i) =d_i^2$. We will aim to apply Theorem \ref{thm:f-factor} with $a=1$ and $b=d_1^2\ge 2$. Note that $\sum_{v\in V(G)}f(v) \equiv 0 \pmod{2}$ since $(d_1^2, \ldots, d_n^2)$ is graphic. By assumption, $b< \sqrt{n/2}$, and so
$$\tfrac{(a+b)^2}{2}< (1+\sqrt{n/2})^2/2 \leq n,$$
with the last inequality following since $n\geq 3$. Theorem~\ref{thm:f-factor} will therefore give our desired conclusion provided that %either
$$\bi(\overline{G}) \ge \frac{(a+b-1)(n-1)}{an-(a+b)+3} =\frac{b(n-1)}{n-b+2}.$$
%or
%$$\delta(\overline{G}) \ge \frac{bn-2}{a+b}=\frac{bn-2}{b+1}.$$

Let $X\subseteq V(\overline{G})$ with $N_{\overline{G}}(X) \ne V(\overline{G})$. So there is a vertex $y\not\in N_{\overline{G}}(X)$ which must have all of its neighbours outside of $X$ in $\overline{G}$. By definition of $\overline{G}$, we know that $\delta(\overline{G})=n-1-d_1^1 >n-1- \frac{n/2}{b}$. Hence we must have $|X|< \frac{n}{2b}$. Using this, we get
$$
\frac{|N_{\overline{G}}(X)|}{|X|} > \frac{n-1- \frac{n/2}{b}}{|X|}
> \frac{n-1- \frac{n/2}{b}}{ \frac{n}{2b}}= \frac{2b(n-1)-n}{n}.
$$
It suffices to prove that:
\begin{eqnarray*}
&&\frac{2b(n-1)-n}{n} \geq  \frac{b(n-1)}{n-b+2} \\
&\Leftrightarrow& bn^2-2b^2n+4bn +2b^2-4b-n^2-2n \geq 0.
\end{eqnarray*}
To see this, by noting $ 2\le b< \sqrt{\frac{n}{2}}$, we get
\begin{eqnarray*}
    &&bn^2-2b^2n+4bn +2b^2-4b-n^2-2n \\
    &=& (bn^2-2b^2 n -n^2)+(4bn-2n)+(2b^2-4b) \\
    &\ge & (2n^2-n^2-n^2)+(8n-2n)+(4b-4b) \ge 0.
\end{eqnarray*}
\end{proof}

\section{Proof of Theorem \ref{t.extremal}}

Given two matchings $M, N \in \family$, we say that $N \preceq M$ if every degree sequence $(d_1, \ldots, d_n)$ which can realize $M$ can also realize $N$. Given this, we can restate Theorem \ref{t.extremal} equivalently as follows.

\begin{theorem}\label{t.extremaref} For any $M \in \family$,
$\mmin \preceq M \preceq \mmax.$
\end{theorem}

A \emph{preorder}  is a binary relation that is reflexive and transitive; note that  $\preceq$ is a preorder on $\family$.
Theorem \ref{t.extremaref} asserts that the preorder $(\family,\preceq)$ has \emph{minimum} and \emph{maximum} elements, respectively, $\mmin$ and $\mmax$.

We start our work towards Theorem \ref{t.extremaref} with a helpful example, namely the family $\mathcal{M}_4$. This family consists of three members:
\[
M^1=\{(1,4),(2,3)\}, \, M^2=\{(1,3),(2,4)\}, \text{ and }   M^3=\{(1,2),(3,4)\}.
\]
Theorems \ref{Lov} and \ref{EG} tell us which degree sequences of length four can realize at least one 1-factor, and we can manually check which of them can realize each of $M^1, M^2, M^3$. We get that the degree sequences are:
\begin{itemize}
    \item $(1,1,1,1),(2,2,2,2),(3,3,3,3)$, each of which can realize any of $M^1, M^2, M^3$ ;
    \item $(2,2,1,1),(3,3,2,2)$, each of which can realize $M^1$ or $M^2$ but not $M^3$; and
    \item $(3,2,2,1)$, which can realize $M^1$ but not $M^2$ or $M^3$.
\end{itemize}
This tells us that $M^1 \preceq M^2 \preceq M^3$. In particular,
this confirms Theorem \ref{t.extremaref} for $n=4$, with $M^1=M^{-}$ and $M^3=M^{+}$.

Given $N,M \in \family$, we say that $N$ is a \emph{switch} of $M$ if the symmetric difference $M \Delta N$
consists of the single 4-cycle $(w,x,y,z)$ for some $w, x, y, z \in[n]$ with $w < x < y < z$ and either:
\begin{enumerate}[(1)]
        \item  $M \setminus N = \{(w,x),(y,z)\}$ and $N \setminus M = \{(w,y),(x,z))\}$, or
        \item $M \setminus N = \{(w,y),(x,z)\}$ and $N \setminus M = \{(w,z),(x,y))\}$, or
        \item $M \setminus N = \{(w,x),(y,z)\}$ and $N \setminus M = \{(w,z),(x,y))\}$;
    \end{enumerate}
we say such switches are of \emph{type (1)}, \emph{type (2)}, or \emph{type (3)}, respectively.
We can think of $\mathcal{M}_4$ as providing canonical examples of switches: in $\mathcal{M}_4$, a switch of type (1) corresponds to moving from $M^3$ to $M^2$; a switch of type (2) corresponds to moving from $M^2$ to $M^1$, and; a switch of type (3) corresponds to moving from $M^3$ to $M^1$. See Figure \ref{M4}. In general, a switch of type (3) can always be obtained by a switch of type (1) followed by a switch of type (2).

\begin{figure}[htb]
    \centering
    \includegraphics[height=5.5cm]{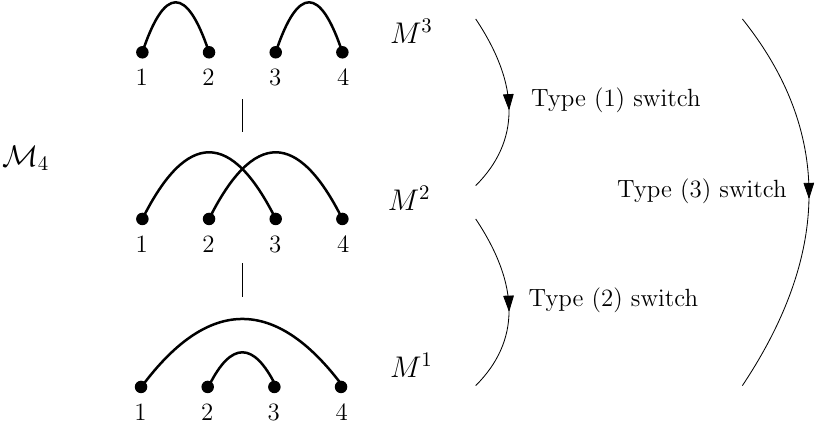}
    \caption{A depiction of $(\mathcal{M}_4, \preceq)$ and canonical examples for our three types of switches.}
    \label{M4}
\end{figure}

For any (not necessarily perfect) matching $M$ on $[n]$, define $\phi(M)$ by
\[
\phi(M):= \sum_{(u,v) \in M} 2^{u+v}.
\]
We prove that ``switches make $\phi$ smaller'', and that ``switches make the matching go down in the preorder''.

\begin{lemma}\label{l.monovariant}
    Let $N, M \in\family$, and suppose that $N$ is a switch of $M$. Then
    \begin{enumerate}
    \item[(a)]  %l.monovariant
    $\phi(N)<\phi(M)$, and;
    \item[(b)] %\label{p.1switchgoesdown}
$N \preceq M$.
\end{enumerate}
\end{lemma}

\begin{proof}
    Let $w, x, y, z\in[n]$ with $w < x <  y < z$ be the four vertices involved in the switch.

    \noindent (a) We know that
    \[
    2^{w+x} + 2^{y +z} > 2^{w + y} + 2^{x+z}
                        > 2^{w+z} +2^{x+y},
    \]
    where the first inequality used $2 \cdot 2^{y+z-1} \geq 2^{w+y} + 2^{x+z}$ and the second inequality used
    $2 \cdot 2^{x+z-1} \geq 2^{w+z} + 2^{x+y}$. From this we deduce that $\phi(M \setminus N) > \phi(N \setminus M)$, as
    \begin{itemize}
        \item if $N$ is a switch of $M$ of type (1) then $\phi(M \setminus N) = 2^{w+x} + 2^{y+z}$ and $\phi(N \setminus M) = 2^{w+y} + 2^{x+z}$;
        \item if $N$ is a switch of $M$ of type (2) then $\phi(M \setminus N) = 2^{w+y} + 2^{x+z}$ and $\phi(N \setminus M) = 2^{w+z} + 2^{x+y}$; and
        \item if $N$ is a switch of $M$ of type (3) then $\phi(M \setminus N) = 2^{w+x} + 2^{y+z}$ and $\phi(N \setminus M) = 2^{w+z} + 2^{x+y}$.
    \end{itemize}
It follows that $\phi(M) = \phi(M \cap N) + \phi(M \setminus N)
    > \phi(M \cap N) + \phi(N \setminus M) = \phi(N)$.

\noindent (b) We assume
    that $M \setminus N= \{(w,x),(y,z)\}$ and $N \setminus M = \{(w,y),(x,z)\}$ (the arguments for switches of types (2) and (3) are similar). Suppose that $(d_1, d_2, \dots, d_n)$ is a degree sequence which can realize $M$; we must show that it can also realize $N$.

   Suppose that $G$ realizes $(d_1, d_2, \dots, d_n)$ and contains $M$. Certainly, if $G$ already contains both edges $(w,y)$ and $(x,z)$ of $N \setminus M$, then $G$ already contains all of $N$, and we are done. In fact, if $G$ contains \emph{neither} $(w,y)$ nor $(x,z)$, then we can alternate along this 4-cycle to get $H=G - \{(w,x),(y,z)\} + \{(w,y),(x,z)\}$, and we are again done. So we may assume that $G$ has exactly one of $(w,y)$ and $(x,z)$.

Suppose first that $(w,y)$ is in $G$ but $(x,z)$ is not. Then $x \not\sim z \sim y$, yet $x < y$ implies $d_x \geq d_y$, so there must be some other vertex $q$ with $x \sim q \not\sim y$ in $G$. We now alternate along the cyclic sequence $xzyqx$ by letting $H=G -\{(x,q),(y,z)\} + \{(x,z),(y,q)\}$. Then $H$ also has degree sequence $(d_1, \dots, d_n)$ and now contains $(x,z)$. Since $N$ is a matching, the removed edges $(x,q)$ and $(y,z)$ are not in $N$, so as all of $N \setminus \{(x,z)\} \subseteq M \cup \{(w,y)\}$ was initially in $G$ we now have $N \subseteq E(H)$.

We may now assume that $(x,z)$ is in $G$ but $(w,y)$ is not. Then $w \not\sim y \sim z$, yet $w < z$ implies $d_w \geq d_z$, so there must be some $q$ with $w \sim q \not\sim z$ in $G$. In this case, $H:=G -\{(w,q),(y,z)\} + \{(w,y),(z,q)\}$ contains $N$ as desired.
\end{proof}

We can now prove our main result of this section, namely that $\mmin$ and $\mmax$ are the minimum and maximum elements of $(\family, \preceq)$, respectively.

\begin{proof} \emph{(Theorem \ref{t.extremaref})} Let $M \in \family$.
We will prove that
\begin{enumerate}
        \item[(i)]
$\mmin$ is obtainable from $M$ by a sequence of switches, and
\item[(ii)]   $M$ is obtainable from $\mmax$
    by a sequence of switches.
\end{enumerate}
By Lemma \ref{l.monovariant}(b) (and transitivity of $\preceq$), statements (i) and (ii) will imply that $M \preceq \mmax$, and $\mmin \preceq M$, desired.

We first show part (i). Starting with $N_1=M$, inductively construct a sequence of matchings $M=N_1, N_2, \dots$ in $\family$ as follows. Given $N_k$, if there is no matching which is a switch of $N_k$, the sequence terminates with $N_k$. Otherwise, we let $N_{k+1}$ be a switch of $N_k$ (of any kind).

By Lemma \ref{l.monovariant}(a), the sequence of positive integers $(\phi(N_j))_{j \geq 1}$
is strictly decreasing, and so the above sequence must terminate with some $N_k$ which has no switch. We argue that in fact $N_k=\mmin$. To do this, we associate to every edge $e=(w,x)$ with $w<x$ a corresponding interval $[w,x]$ of real numbers, which we denote by $I_e$.

\begin{claim}\label{eeprime}
    For any two edges $e,e'$ of $N_k$, either $I_e \subset I_{e'}$ or $I_{e'} \subset I_e$.
\end{claim}

\begin{proofc}
Suppose to the contrary that $e,e'\in N_k$ are two edges for which $I_e$ and $I_{e'}$ are not nested. We will define a new matching $\tilde{N}$ according to whether $I_e$ and $I_{e'}$ intersect. In the case that $I_e$ and $I_{e'}$ are disjoint, without loss of generality, we may write $e=(w,x)$ and $e'=(y,z)$ where $w<x<y<z$.
Otherwise, $I_e$ and $I_{e'}$ intersect. As their endpoints are distinct (since $N_k$ is a matching), but they are not nested, the two intervals contain exactly one of each other's endpoints, so their corresponding edges $\{e,e'\}$ are of the form $\{(w,y),(x,z)\}$ where $w<x<y<z$. In either case, we let $\tilde{N}=N_k \setminus\{e,e'\} \cup \{(w,z),(x,y)\}$. Then $\tilde{N}$ is a switch of $N_k$ (of type (3) or (2) in the respective cases above), contradicting $N_k$ having no switch.
\end{proofc}

Note that Claim \ref{eeprime} implies that $N_k=\mmin$, concluding the proof of (i). The proof of (ii) is similar, but included for sake of completeness. We define a sequence of matchings $M=M^1, M^2, \dots$ in $\family$ such that each $M^j$ is a switch of $M^{j+1}$. Then $(\phi(M^j))_{j\geq 1}$ is strictly increasing by Lemma \ref{l.monovariant}, but bounded above (as there are only finitely many matchings in $\family$), so must terminate with some $M^\ell$ which is not a switch of any other matching. The following claim implies that $N_k=\mmax$, and hence completes our proof.

\begin{claim}
    For any $e,e' \in M_\ell$, the intervals $I_e$ and $I_{e'}$ are disjoint.
\end{claim}
\begin{proofc}
Suppose for sake of contradiction $I_e \cap I_{e'} \neq \emptyset$, and let the set of endvertices of $e$ and $e'$ be $\{w,x,y,z\}$ where $w < x < y < z$. In the case that $I_e$ and $I_{e'}$ are nested, their corresponding edges $\{e,e'\}$ are of the form $\{(w,z),(x,y)\}$. Or, they contain exactly one of each other's endpoints, and so $\{e,e'\}=\{(w,y),(x,z)\}$. Either way, let $\tilde{M}=M_\ell \setminus\{e,e'\} \cup \{(w,x),(y,z)\}$. Then $M_\ell$ is a switch of $\tilde{M}$ (of type (3) or (1) respectively), again a contradiction.
\end{proofc}
\end{proof}

Observe that the above arguments gives rise to a polynomial-time (in fact $O(n^3)$) algorithm for constructing a realization of any $M\in\family$ given a realization of $\mmax$.

We close this section by making two conjectures about the preorder $(\family, \preceq)$.  Given $N, M \in\family$ with $N \preceq M$, it is not necessarily true that $N$ is a switch of $M$. But we believe that $N$ must be obtainable from $M$ via some sequence of switches.

\begin{conjecture}\label{c.sameposet}
Suppose $N, M \in \M_n$ have $N \preceq M$. Then there exists a sequence
\[
N=M_1 \preceq M_2 \preceq \dots \preceq M_{k-1} \preceq M_k=M
\]
such that $M_j$ is a switch of $M_{j+1}$, for each $j$.
\end{conjecture}

Note that Conjecture \ref{c.sameposet} is effectively a converse of Lemma \ref{l.monovariant}(b). If Conjecture \ref{c.sameposet} is true then it would also imply the following.

\begin{conjecture}\label{c.poset}
   The preorder $(\M_n, \preceq)$ is in fact a poset. That is, it is antisymmetric ($M \preceq N$ and $N \preceq M$ implies $M = N$).
\end{conjecture}

To see that Conjecture \ref{c.sameposet} implies Conjecture \ref{c.poset}, let $M, N$ be matchings on $[n]$ with $N \preceq M \preceq N$. Then, assuming Conjecture \ref{c.sameposet},
\[
N=M_1 \preceq M_2 \preceq \dots \preceq M_{k-1} \preceq M_k=M \preceq M_{k+1} \preceq \dots \preceq M_{t-1} \preceq M_t=N
\]
for some sequence of ordered matchings $\{M_j\}$ such that $M_j$ is a switch of $M_{j+1}$, for each $j$. Suppose for sake of contradiction that $N \neq M$, so that the length $t$ of this sequence is strictly more than 1.
Then by $t-1$ applications of Lemma \ref{l.monovariant}(a),
\[
\phi(N)=\phi(M_1) < \phi(M_2) <  \dots < \phi(M_{t-1}) < \phi(M_t)=\phi(N),
\]
a contradiction.

It is easy to see that $(\M_4, \preceq)$ is antisymmetric, and hence is a poset (see Figure \ref{M4}). In fact $(\M_4, \preceq)$ is total order, although this is not true in general, and indeed $(\family, \preceq)$ is not a total order for even $n > 4$.
As an example, Figure \ref{fig:M6} depicts $(\M_6, \preceq)$ and there the two matchings $\{(1,6), (2,4), (3,5)\}$ and $\{(1,5),(2,6),(3,4)\}$ immediately above $M^-$ that are incomparable. To see this, note the degree sequence $(5,3,3,3,3,1)$ realizes the former but not the latter, while $(5,5,3,3,2,2)$ realizes the latter but not the former.

\begin{figure}
    \begin{center}
\begin{tikzpicture}[ultra thick, scale=0.7]

\node (a) at (0,18) {};
\node (b) at (-3,15) {};
\node (c) at (3,15) {};
\node (d) at (-6,12) {};
\node (e) at (0,12) {};
\node (f) at (6,12) {};
\node (g) at (-6,9) {};
\node (h) at (0,9) {};
\node (i) at (6,9) {};
\node (j) at (-6,6) {};
\node (k) at (0,6) {};
\node (l) at (6,6) {};
\node (m) at (0,3) {};
\node (n) at (6,3) {};
\node (o) at (0,0) {};

\foreach \x in {a,b,c,d,e,f,g,h,i,j,k,l,m,n,o} {
\foreach \y in {0,1,2,3,4,5}{
\draw [fill=black,draw=none] (\x)+(\y*0.5,0) circle (0.08cm);
};
};

\node [left] at (a) {$M^+ \;$};
\node [left] at (o) {$M^- \;$};

\draw (0,18) to [out=90,in=90] (0.5,18);
\draw (1,18) to [out=90,in=90] (1.5,18);
\draw (2,18) to [out=90,in=90] (2.5,18);

\draw (a)+(0.5,-0.5) to (-1.5,16.0);
\draw (a)+(2.0,-0.5) to (4.0,16.0);

\draw (3,15) to [out=90,in=90] (3.5,15);
\draw (4,15) to [out=90,in=90] (5.0,15);
\draw (4.5,15) to [out=90,in=90] (5.5,15);

\draw (-3,15) to [out=90,in=90] (-2.0,15);
\draw (-2.5,15) to [out=90,in=90] (-1.5,15);
\draw (-1,15) to [out=90,in=90] (-0.5,15);

\draw (c)+(2,-0.5) to (7,13.0);
\draw (c)+(0.5,-0.5) to (2,13.0);
\draw (b)+(2,-0.5) to (0.5,13.0);
\draw (b)+(0.5,-0.5) to (-4.5,13.0);

\draw (6,12) to [out=90,in=90] (6.5,12);
\draw (7,12) to [out=90,in=90] (8.5,12);
\draw (7.5,12) to [out=90,in=90] (8.0,12);

\draw (0,12) to [out=90,in=90] (1.0,12);
\draw (0.5,12) to [out=90,in=90] (2.0,12);
\draw (1.5,12) to [out=90,in=90] (2.5,12);

\draw (-6,12) to [out=90,in=90] (-4.5,12);
\draw (-5.5,12) to [out=90,in=90] (-5.0,12);
\draw (-4.0,12) to [out=90,in=90] (-3.5,12);

\draw (d)+(1.25,-0.5) to (-4.75,10.0);
\draw (e)+(1.25,-0.5) to (1.25,10.0);
\draw (f)+(1.25,-0.5) to (7.25,10.0);
\draw (d)+(2.5,-0.5) to (0,10.0);
\draw (e)+(0.0,-0.5) to (-3.5,10.0);
\draw (e)+(2.5,-0.5) to (6.0,10.0);
\draw (f)+(0.0,-0.5) to (2.5,10.0);

\draw (6,9) to [out=90,in=90] (7.0,9);
\draw (6.5,9) to [out=90,in=90] (8.5,9);
\draw (7.5,9) to [out=90,in=90] (8.0,9);

\draw (0,9) to [out=90,in=90] (1.5,9);
\draw (0.5,9) to [out=90,in=90] (2.0,9);
\draw (1.0,9) to [out=90,in=90] (2.5,9);

\draw (-6,9) to [out=90,in=90] (-4.0,9);
\draw (-5.5,9) to [out=90,in=90] (-5.0,9);
\draw (-4.5,9) to [out=90,in=90] (-3.5,9);

\draw (g)+(1.25,-0.5) to (-4.75,7.0);
\draw (g)+(2.5,-0.5) to (0,7.0);
\draw (h)+(0.0,-0.5) to (-3.5,7.0);
\draw (h)+(2.5,-0.5) to (6.0,7.0);
\draw (i)+(0.0,-0.5) to (2.5,7.0);
\draw (i)+(1.25,-0.5) to (7.25,7.0);

\draw (6,6) to [out=90,in=90] (7.5,6);
\draw (6.5,6) to [out=90,in=90] (8.5,6);
\draw (7.0,6) to [out=90,in=90] (8.0,6);

\draw (0,6) to [out=90,in=90] (2.5,6);
\draw (0.5,6) to [out=90,in=90] (1.0,6);
\draw (1.5,6) to [out=90,in=90] (2.0,6);

\draw (-6,6) to [out=90,in=90] (-4.0,6);
\draw (-5.5,6) to [out=90,in=90] (-4.5,6);
\draw (-5.0,6) to [out=90,in=90] (-3.5,6);

\draw (j)+(2,-0.5) to (0.0,4.0);
\draw (j)+(2.5,-0.5) to (6.0,4.0);
\draw (k)+(1.25,-0.5) to (1.25,4.0);
\draw (l)+(0.5,-0.5) to (2.5,4.0);
\draw (l)+(1.25,-0.5) to (7.25,4.0);

\draw (0,3) to [out=90,in=90] (2.5,3);
\draw (0.5,3) to [out=90,in=90] (1.5,3);
\draw (1.0,3) to [out=90,in=90] (2.0,3);

\draw (6,3) to [out=90,in=90] (8.0,3);
\draw (6.5,3) to [out=90,in=90] (8.5,3);
\draw (7.0,3) to [out=90,in=90] (7.5,3);

\draw (m)+(1.25,-0.5) to (1.25,1.0);
\draw (n)+(0.0,-0.5) to (2.5,1.0);

\draw (0,0) to [out=90,in=90] (2.5,0);
\draw (0.5,0) to [out=90,in=90] (2.0,0);
\draw (1.5,0) to [out=90,in=90] (1.0,0);

\node [below] at (0,0) {1};
\node [below] at (0.5,0) {2};
\node [below] at (1.0,0) {3};
\node [below] at (1.5,0) {4};
\node [below] at (2.0,0) {5};
\node [below] at (2.5,0) {6};
\end{tikzpicture}
\end{center}
\caption{The preorder $(\M_6,\preceq)$ 
}
\label{fig:M6}
\end{figure}
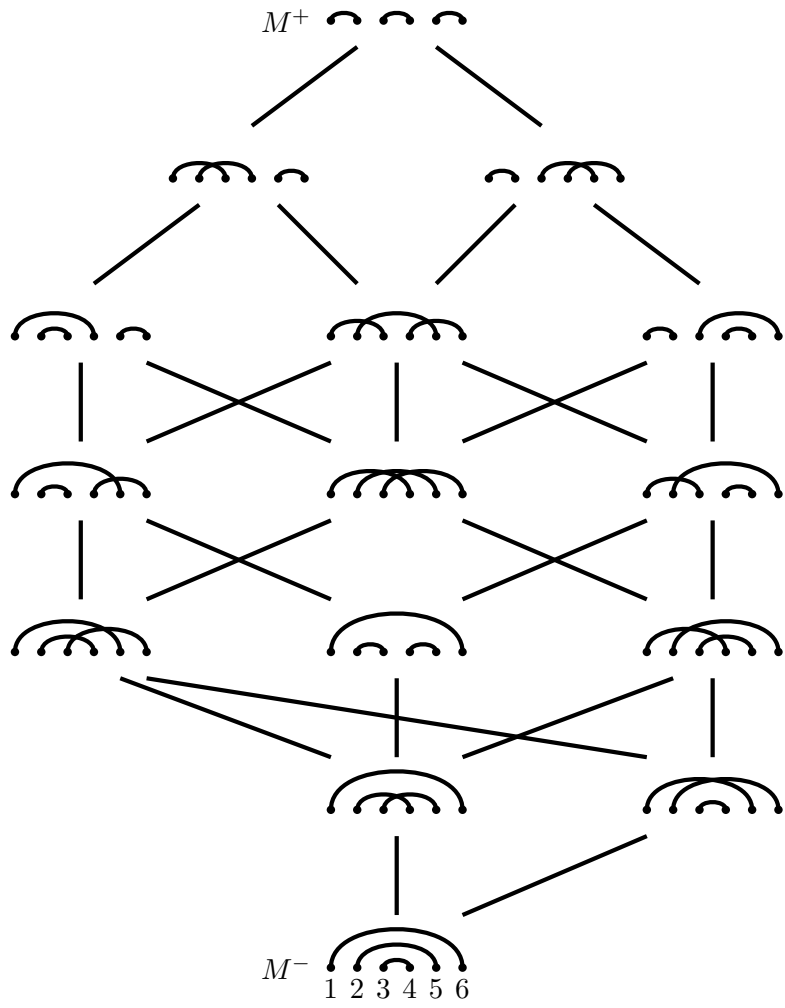

\section{Labelled h-factors}

The forwards direction of the proof of Theorem \ref{t.SpecialMatch}, as written in Section 2, is an easy argument that we can generalize to $h$-factors. To this end, suppose that $(d_1, \ldots, d_n)$ realizes the $h$-factor {$\{(1, 2,\ldots, h+1),  \ldots, (n-h, \ldots, n)\}$ via the graph $G$, where we use  $(i,i+1, \ldots, i+h)$ for each $i\in \{1,\ldots, n-h\}$ to denote the complete graph on $\{i, \ldots, i+h\}$.
Since $G$ is graphic we must have $\sum_{i=1}^n d_i$ is even, and since our $h$-factor is composed of exactly $n/(h+1)$ copies of $K_{h+1}$, we know that $n$ must be a multiple of $h+1$. Consider the graph $G'$ obtained from $G$ by deleting  all the edges in our $h$-factor that are induced by $[k+1, n]$.  Then $G'$ must satisfy the inequality of Erd\H{o}s and Gallai for $k$. This amounts precisely to (\ref{EGsumFactor}) below, motivating the following conjecture.

\begin{conjecture}\label{SpecialFactor} Let $n\in\mathbb{N}$ and $d_1 \geq d_2 \geq \dots d_n\geq 1$ be integers, and let $h$ be a positive integer. The sequence $(d_1, d_2, \ldots, d_n)$ can realize the $h$-factor $\{(1, 2,\ldots, h+1),  \ldots, (n-h, \ldots, n)\}$  iff $\sum_{i=1}^n d_i$ is even, $n$  is a multiple of $h+1$, and for every $k\in[n]$:
	\begin{equation}\label{EGsumFactor}\tag{**}
		 \sum_{i=1}^kd_i \le k(k-1)+\sum_{i=k+1}^{k+1+h-s}\min\{d_i-h+s, k\}+\sum_{i=k+1+h-s+1}^n\min\{d_i-h, k\},
	\end{equation}
where $s\in \{0, \ldots, h\}$ and $s\equiv k \pmod{h+1}$.
\end{conjecture}

In order to prove Conjecture \ref{SpecialFactor}, it seems that some new ideas beyond our proof of Theorem \ref{t.SpecialMatch} are needed. In fact, the induction argument of our proof would appear to mostly extend to this more general setting, except for the crucial Claim 4.

Note that in the statement of Conjecture \ref{SpecialFactor}, we have not named an $h$-factor analog for $\mmax$, and indeed this is because the $h$-factor analog of Theorem \ref{t.extremaref} is false:
there can be multiple maximum and minimum elements even when $h=2$, as shown in Examples \ref{ex.multimax} and \ref{ex.multimin} below.
For both of these examples, we let $\G^n$ be the family of all $2$-regular graphs on the vertex set $[n]$, and we define the preorder $\preceq$ on $\G_{}^n$ by saying two graphs $K,F \in \G^n$ have $K \preceq F$ if and only if every degree sequence which realizes $F$ also realizes $K$.

\begin{example}\label{ex.multimax}
We claim that the two 5-cycles $C=135421$ and $C'=123541$ are both maximum elements of the preorder $(\G^5,\preceq)$ (even though $C'$ is a switch of $C$ of type (2)). To see this, first list the degree sequences which realize \emph{at least one} of the twelve $C_5$'s in $\G^5$. By Theorem \ref{Lov}, these are just the graphic sequences $D$ of length 5 for which $D-(2,2,2,2,2)$ is also graphic:
\begin{itemize}
    \item $(4,4,4,4,4),(4,3,3,3,3),(3,3,3,3,2), (2,2,2,2,2)$,
    \item $(4,4,4,3,3),(4,4,3,3,2),(3,3,2,2,2)$, and $(4,3,3,2,2)$.
\end{itemize}
The first 4 degree sequences realize all of $\G^5$, and none of the latter 4 realize $C$ or $C'$. So not only are $C$ and $C'$ both maximum elements, but also, $C \preceq C' \preceq C$,
%(they are equivalent),
that is to say, $(\G_5, \preceq)$ is not antisymmetric (so not a poset). It turns out here that $135241 \in \G^5$ is a minimum element, but we do not have a sole minimum in general.
\end{example}

\begin{example}\label{ex.multimin} We claim that in $\G^{12}$, there are two minimal elements, namely the two 2-factors whose respective components are:
\begin{enumerate}[(a)]
    \item the three 4-cycles $1 \;12\; 2\; 11\; 1,\; 3\; 10\; 4\; 9\; 3,\; 5\; 8\; 6\; 7\; 5$; and
    \item the two 6-cycles
    $1\; 12\; 2\; 10\; 3\; 11\; 1,\;
    4\; 9\; 5\; 7\; 6\; 8\; 4$.
\end{enumerate}
For (a), let $G$ be a graph with degree sequence $(11,11,9,9,7,7,6,6,4,4,2,2)$ (in fact there is only one). Since $d_1=d_2=11$,  vertices  1 and 2 are the only neighbors of $11$ and $12$, thus any 2-factor of $G$ contains the cycle $1 \;12\; 2\; 11\; 1$.  More rounds of a similar argument shows that the copy of $3C_4$ in (a)  is the \emph{only} 2-factor of $G$.
It can likewise be shown that the copy of $2C_6$ from (b) is only 2-factor which realizes all of
\begin{eqnarray*}
   && (11,11,10,8,8,7,6,6,5,3,3,2), \\
   && (11,3,3,3,3,3,3,3,3,3,2,2), \quad  \text{and}\\
   && (11,11,10,8,6,6,6,5,5,3,3,2).
\end{eqnarray*}
\end{example}

Although Conjecture \ref{SpecialFactor} does not pair with an analog of Theorem \ref{t.extremaref}, if Conjecture \ref{SpecialFactor} is true then it implies a sufficient condition for sequences realizing $h$ disjoint 1-factors, as we will now explain.

Let $H_1$ and $H_2$ be two disjoint copies of $K_{2k+1}$ for some integer $k\ge 1$, and let $M_i=\{a_1^ib_1^i, \ldots, a_k^ib_k^i\}$ be a near perfect matching of $H_i$ (i.e. a matching saturating all but exactly one vertex of $H_i$) for $i=1,2$.
We define $H_1*H_2$ as the graph obtained from $(H_1-M_1)\cup (H_1-M_2)$ by adding a matching between  $\{a_j^1, b_j^1\}$ and $\{a_i^2, b_j^2\}$ for each $j\in [1,k]$.
As $E(H_i)$ can be decomposed into $(2k+1)$ near perfect matchings, it follows that $H_1*H_2$ has a 1-factorization, that is, a decomposition of $E(H_1*H_2)$ into $2k$ perfect matchings.

\begin{lemma}\label{lem:2-switch} Let $n\in\mathbb{N}$ be an even integer, let $d_1 \geq d_2 \geq \dots d_n\geq 1$ be integers, and suppose that $G$ realizes $(d_1, \ldots, d_n)$.
Suppose that $G$ has two vertex-disjoint subgraphs $H_1$ and $H_2$ each of which is isomorphic to $K_{2k+1}$ for some integer $k$.
Then there is a sequence of switches which maintains the degree sequence of $G$ and keeps the subgraph $G-V(H_1\cup H_2)$ unchanged,  but which transforms $H_1\cup H_2$ into a graph on $V(H_1\cup H_2)$  that contains  $H_1*H_2$ as a subgraph.
\end{lemma}
\proof
The proof proceeds by induction on $k$.
Let $v_1v_2v_3$ be a triangle  in $H_1$ and $u_1u_2u_3$ be a triangle in $H_2$.
If there is a matching of size 2 in the bipartite graph $G[\{v_1,v_2, v_3\},  \{u_1,u_2, u_3\}]$,
say $v_1u_1, v_2u_2$, then we are done when $k=1$;
and are done by applying the induction hypothesis on  $H_1-\{v_1,v_2\}$ and $H_2-\{u_1,u_2\}$
if $k\ge 2$.
Thus we assume that $G[\{v_1,v_2, v_3\},  \{u_1,u_2, u_3\}]$
has no matching of size 2.  Then we can find, say, $v_1, v_2$ and $u_1, u_2$
such that $G[\{v_1,v_2\},  \{u_1,u_2\}]$  is an empty graph.
Then we simply replace the edges $v_1v_2$ and $u_1u_2$
by $v_1u_1$ and $v_2u_2$.
Again,  we are done when $k=1$;
and are done by applying the induction hypothesis on  $H_1-\{v_1,v_2\}$ and $H_2-\{u_1,u_2\}$
if $k\ge 2$.
\qed

In the following Corollary we assume the same conditions as in Conjecture \ref{SpecialFactor}; if that conjecture holds then we can also realize a packing of perfect matchings.

\begin{corollary}[Assuming Conjecture~\ref{SpecialFactor}]\label{corSpecialFactor}
Let  $h, n \in \mathbb{N}$ be  integers and  $d_1 \geq d_2 \geq d_n\geq h$ be integers. Then  $(d_1, d_2, \ldots, d_n)$ can realize
$h$ disjoint perfect matchings in the same graph if $\sum_{i=1}^n d_i$ is even, $n$  is a multiple of $h+1$, and (**) holds for every $k\in[n]$.
\end{corollary}

\proof  Applying  Conjecture~\ref{SpecialFactor}, we get a realization of $(d_1, d_2, \ldots, d_n)$ that contains the
$h$-factor $H:=\{(1, 2,\ldots, h+1),  \ldots, (n-h, \ldots, n)\}$. When $h$ is odd, the subgraph $H$ decomposes into $h$ disjoint perfect matchings.
Thus we assume that $h$ is even. Then as $n$ is even, the total number of components of $H$, which is equal to $n/(h+1)$, is even. Let $H_1, \ldots, H_{n/(h+1)}$ be all the components of $H$.
We apply Lemma~\ref{lem:2-switch}  on  each  pair of $H_{2i-1}$ and $H_{2i}$, $i\in \{1, \ldots, \frac{n}{2(h+1)}\}$.
Then we get a realization $G'$ of $(d_1, d_2, \ldots, d_n)$  that contains $\bigcup_{i=1}^{\frac{n}{2(h+1)}} (H_{2i-1}*H_{2i})$
as a spanning subgraph. As $\bigcup_{i=1}^{\frac{n}{2(h+1)}} (H_{2i-1}*H_{2i})$
has a decomposition into perfect matchings, it follows that $G'$ has $h$ disjoint perfect matchings.
\qed

Brualdi \cite{brualdi1978problemes} and Busch, Ferrera, Hartke, Jacobsen, Kaul, and West \cite{busch2012packing} have a conjecture which states that for a degree sequence $(d_1, \ldots, d_n)$ where $n$ is even, it can realize
$h$ disjoint perfect matchings in the same graph if and only if $(d_1-h,\ldots, d_n-h)$ is graphic. This conjecture is open in general; see Shook \cite{Shook} for a discussion of partial results. The sufficient condition of Corollary \ref{corSpecialFactor} (and Conjecture \ref{SpecialFactor}) provides some support for this line of work.

\bibliographystyle{abbrv}
\bibliography{biblio}

\end{document}